\documentclass[11pt,a4paper,twoside]{amsart}
\usepackage{tcolorbox}
\usepackage{fancyhdr}
\usepackage{fancyhdr}
\usepackage{amsmath, amsfonts, amssymb, amsthm}
\usepackage[english]{babel}
\usepackage[all]{xy}
\usepackage{graphicx}
\usepackage[T1]{fontenc}
\usepackage[utf8]{inputenc}
\usepackage{color}
\usepackage[a4paper,top=3cm,bottom=3cm,left=3cm,right=3cm]{geometry}
\usepackage{tikz}
\usetikzlibrary{matrix, arrows, decorations.pathmorphing}
\usepackage{plain}
\usepackage{url}
\usepackage{hyperref}
\usepackage{microtype}

\title{Componentwise Linearity Under Square-Free Gr\"obner Degenerations}
\author{Hongmiao Yu}
\email{yu@dima.unige.it}
\address{Dipartimento di Matematica, Universit\'a di Genova, Italy}
 \thanks{The author is supported by PRIN  2020355BBY ``Squarefree Gr\"obner degenerations, special varieties and related topics".} 
  \date{}

\newtheoremstyle{break}%
{}{}%
{\slshape}{}%
{\bfseries}{.}%
{5pt}{}

\theoremstyle{definition}
\newtheorem{definition}{Definition}[section]

\newtheorem{example}[definition]{Example}

\newtheorem{notation}[definition]{Notation}

\newtheorem{remark}[definition]{Remark}

\theoremstyle{break}
\newtheorem{lemma}[definition]{Lemma}
\newtheorem{proposition}[definition]{Proposition}

\newtheorem{theorem}[definition]{Theorem}
\newtheorem{corollary}[definition]{Corollary}

\newtheorem*{maintheorem*}{Main\ Theorem(cf. Theorem\til\ref{main})}
\newtheorem*{theorem*}{Theorem}
\newtheorem*{theorem3.5*}{Theorem(cf. Theorem\til\ref{thm35})}
\newtheorem*{corollarymain*}{Corollary\til\ref{main}}
\newtheorem*{question*}{Question}

\def\til{~}

\newcommand\tbs[1]{\textsl{\textit{#1}}}  

\def\N{\mathbb N}

\def\Z{\mathbb Z}

\def\mm{\mathfrak m}  

\def\H{H} 

\def\la{\longrightarrow}

\def\:{\colon}

\def\pd{\mathrm{projdim}}  

\def\reg{\mathrm{reg}}

\def\Ext{\mathrm{Ext}}

\def\se{\subseteq}

\def\iso{\cong}

\def\init{\mathrm{in}}

\def\HF{\mathrm {HF}}

\def\HS{\mathrm {HS}}

\def\cocoa{\mbox{\rm 
   C\kern-.13em o\kern-.07 em C\kern-.13em o\kern-.15em A}}

\newcommand\fg[0]{finitely generated{}}

\begin{document}
\begin{abstract}
Using the recent results on square-free Gröbner degenerations by Conca and Varbaro, we proved that if a homogeneous ideal $I$ of a polynomial ring is such that its initial ideal $\init_<(I)$ is square-free and $\beta_0(I)=\beta_0(\init_<(I))$, then $I$ is a componentwise linear ideal if and only if  $\init_<(I)$ is a componentwise linear ideal. In particular,  if furthermore one of $I$ and $\init_<(I)$ is  componentwise linear, then their graded Betti numbers coincide.

\end{abstract}

\maketitle

\section{Introduction}
Throughout this paper,  $R=K[X_1,\dots, X_n]$ is the polynomial ring in $n$ variables over a field $K$ with  $\deg (X_i) = 1$ for each $i=1,\dots, n$ and $\mm=(X_1,\dots, X_n)$  is the unique homogeneous maximal ideal of $R$. 
Let $I$ be a homogeneous ideal of $R$. We denote by $\beta_{i,j}(I)$ the $(i,j)$-th graded Betti number of $I$ and, for each $d\in\Z_+$, we denote by $I_{<d>}$ the ideal generated by all homogeneous polynomials of degree $d$ belonging to $I$.

The notion of componentwise linearity was introduced by Herzog and Hibi  \cite{HH} in 1999:
We say that a homogeneous ideal $I\se R$ has a \tbs{$d$-linear resolution} if $\beta_{i,i+j}(I) = 0$ for all $i$ and for all $j\not=d$.
We say that $I$ is \tbs{componentwise linear} if $I_{<d>}$ has a $d$-linear resolution for all $d\in\Z$.
In particular, if  $I$ has a linear resolution, then it is componentwise linear.

In their paper published in 2020 on square-free Gr\"obner degenerations \cite{CV}, Conca and Varbaro showed that if $I$ is a homogeneous ideal of a polynomial ring and if the initial ideal $\init_<(I)$ is square-free with respect to some term order $<$, then the Castelnuovo-Mumford regularity of $I$ and of $\init_<(I))$ coincide \cite[Corollary 2.7]{CV}. A consequence of this result is the following: if $\init_<(I)$ is square-free, then $I$ has a $d$-linear resolution if and only if $\init_<(I)$ has a $d$-linear resolution. On the other hand, since $\init_<(I)$ and $I$ have the same Hilbert function, if both $I$ and $\init_<(I)$ have  $d$-linear resolutions, using their  Hilbert series:
\begin{eqnarray*}
\frac{\sum_{i=0}^{\pd(I)}(-1)^i\sum_{j\in\Z}\beta_{i,j}(I)t^j}{(1-t)^n}&=&\HS_I(t)\\
&=&\HS_{\init_<(I)}(t)\\
&=&\frac{\sum_{i=0}^{\pd({\init_<(I)})}(-1)^i\sum_{j\in\Z}\beta_{i,j}({\init_<(I)})t^j}{(1-t)^n}
\end{eqnarray*}
we have that the graded Betti numbers $\beta_{i.j}(I)=\beta_{i.j}(\init_<(I))$ for all $i,j\in\Z$.
Therefore, if $\init_<(I)$ is square-free and $I$ has a linear resolution, then  their graded Betti numbers coincide.

In what follows we suppose furthermore that  $<$ is a graded term order. 
Since componentwise linear ideals can be considered as a generalization of ideals with linear resolution, naturally one has some questions: if one of the ideals $I$ and $\init_<(I)$ is a componentwise linear ideal, can we obtain that, under some certain assumptions,  the other one is also componentwise linear? Can we have some information about their graded Betti numbers? One part of these questions has already been answered by Caviglia and Varbaro in \cite{CaV}. They proved that if $\init_<(I)$ is a componentwise linear ideal and if $\beta_0(I)=\beta_0(\init_<(I))$, then $I$ is also a componentwise linear ideal \cite[Theorem 5.4]{CaV}. In this paper, we show that if $\init_<(I)$ is square-free, then the converse of the result of Caviglia and Varbaro also holds, that is, 
\begin{quote}
Assume that $\init_<(I)$ is square-free and $\beta_0(I)=\beta_0(\init_<(I))$. Then $I$ is a componentwise linear ideal if and only if  $\init_<(I)$ is a componentwise linear ideal. In particular,  if furthermore one of $I$ and $\init_<(I)$ is  componentwise linear, we have $\beta_{i,i+j}(\init_<(I))=\beta_{i,i+j}(I)$ for all $i,j$.
\end{quote}

\bigskip

{\bf Acknowledgments}: The author wishes to thank Matteo Varbaro for very helpful discussions related to this paper and Alessio D'Al\`i for his useful suggestions and technical assistance.


\section{The Main Result}


 \begin{notation}
For each $d\in\Z_+$, we denote by $I_{\le d}$ the ideal generated by all homogeneous polynomials of  $I$ whose degree is less than or equal to $d$.
\end{notation}

\begin{lemma}\label{l2}
Let  $d\in\Z_+$. Then following conditions are equivalent: 
\begin{enumerate}
\item[$i$)] $\beta_0(I_{<d>})=\beta_0(\init_<(I_{<d>}))$,
\item[$ii$)]$\init_<(I_{<d>})=\init_<(I)_{<d>}$,
\item[$iii$)]$\init_<(I_{\le d})=\init_<(I)_{\le d}$.
\end{enumerate}
\end{lemma}
\begin{proof}\begin{enumerate}
\item[$i\Rightarrow ii$)]
We always have $\init_<(I)_{<d>}\se\init_<(I_{<d>})$.  If $\beta_0(I_{<d>})=\beta_0(\init_<(I_{<d>}))$, then 
$I_{<d>}$ has a Gr\"obner basis (with respect to $<$) which is a minimal system of generators of $I_{<d>}$. This implies that $\init_<(I_{<d>})$ is generated by monomials of degree $d$. Hence $\init_<(I_{<d>})=\init_<(I_{<d>})_{<d>}\se \init_<(I)_{<d>}\se\init_<(I_{<d>})$, that is, $\init_<(I_{<d>})=\init_<(I)_{<d>}$.
\item[$ii\Rightarrow i$)]
Since $\init_<(I_{<d>})=\init_<(I)_{<d>}$, $\init_<(I_{<d>})$ is generated by monomials of degree $d$ and so $\init_<(I_{<d>})=([\init_<(I_{<d>})]_{d})$.
Hence, using the fact that $I_{<d>}$ and $\init_<(I_{<d>})$ have the same Hilbert function, we obtain
\begin{eqnarray*}
\beta_0(\init_<(I_{<d>}))&=&\dim_K([\init_<(I_{<d>})]_{d}) \\
&=&\HF_{\init_<(I_{<d>})}(d)\\
&=&\HF_{I_{<d>}}(d)\\
&=&\dim_K([I_{<d>}]_{d})\\
&=&\beta_0(I_{<d>}).
\end{eqnarray*}
\item[$ii\Rightarrow iii$)]
One inclusion $\init_<(I)_{\le d}\se \init_<(I_{\le d})$ always holds. Now we assume that $\init_<(I_{<d>})=\init_<(I)_{<d>}$ and we prove $\init_<(I_{\le d})\se\init_<(I)_{\le d}$. If $m\in \init_<(I_{\le d})$ is a monomial such that $\deg(m)=a$, then there is $f\in I_{\le d}$ such that $\init_<(f)=m$ and $\deg(f)=a$.  Since $I_{\le d}$ is generated by homogeneous polynomials,  we may assume that $f$ is homogeneous. If $a\le d$, then  $m\in \init_<(I)_{\le d}$ since $m\in  \init_<(I)$ and $\deg(m)\le d$. Now we assume that $a>d$. 
Since  $\mm I_{<i>}\se I_{<i+1>}$ for each $i$, we have 
\begin{eqnarray*}
(I_{<1>})_a\se(I_{<2>})_a\se \dots\se (I_{<d>})_a.
\end{eqnarray*}
Hence $f\in I_{<d>}$ and so $m\in\init_<(I_{<d>}) =\init_<(I) _{<d>}\se \init_<(I)_{\le d}$.
\item[$iii\Rightarrow ii$)]
Since $<$ is graded, $\init_<(I_{<d>})$ is generated by monomials of degree greater  than or equal to $d$. By our assumption $\init_<(I_{<d>})\se\init_<(I_{\le d})=\init_<(I)_{\le d}$. It follows that $\init_<(I_{<d>})$ is generated by monomials of degree $d$ and so $\init_<(I_{<d>})=\init_<(I_{<d>})_{<d>}\se \init_<(I)_{<d>}$.
\end{enumerate}
\end{proof}

\begin{lemma} \label{l3}
Let $J$ be an ideal of $R$ generated by homogeneous polynomials of the same degree $a$. 
If $\beta_0(J)=\beta_0(\init_<(J))$, 
then for each $d \in\Z_+$ we have
\begin{enumerate}
\item[$i$)] $\init_<(\mm^dJ)=\mm^d\init_<(J)$, and
\item[$ii$)] $\beta_0(\mm^dJ)=\beta_0(\init_<(\mm^dJ))$.
\end{enumerate}
\end{lemma}
\begin{proof}
Notice that the assumption $\beta_0(J)=\beta_0(\init_<(J))$ implies that $J$ has a Gr\"obner basis (with respect to $<$) which is a minimal system of generators of $J$, we denote this system by $\{h_1,\dots, h_r\}$. So $\init_<(J)=(\init_<(h_1),\dots, \init_<(h_r))$ and $\deg(h_i)=a$ for all $i$. 
\begin{enumerate}
\item[$i$)] 
It is clear that $\mm^d\init_<(J)=\init_<(\mm^d)\init_<(J)\se \init_<(\mm^dJ).$ We show that $\init_<(\mm^dJ)\se\mm^d\init_<(J)$. If $m\in\init_<(\mm^dJ)$ is a monomial, then  there exists $f\in \mm^dJ$ such that $m=\init_<(f)$. 
Since $f \in \mm^dJ\se J$, $m=\init_<(f)\in (\init_<(h_1),\dots, \init_<(h_r))$. This implies that there exists a monomial $\mu\in R$ and there exists $i\in\{1, \dots, r\}$ such that $m=\mu\init_<(h_i)$. Since  $<$ is graded, $\deg(m)\ge d+a$, and so $\deg(\mu)\ge d$.  It follows that $m=\mu\init_<(h_i)\in\mm^d\init_<(J).$ 
\item[$ii$)] 
Since $\init_<(\mm^dJ)=\mm^d\init_<(J)=\mm^d(\init_<(h_1),\dots, \init_<(h_r))$  and $\mm^dJ$ is generated by monomials of degree $d+a$, we have $\init_<(\mm^dJ)=([\mm^d\init_<(J)]_{d+a})$.  Similarly, since $\mm^dJ$  is generated by homogeneous polynomials of degree $d+a$, we have $\mm^dJ=([\mm^dJ]_{d+a})$. Therefore,
\begin{eqnarray*}
\beta_0(\init_<(\mm^dJ))&=&\dim_K([\mm^d\init_<(J)]_{d+a}) \\
&=&\dim_K([\init_<(J)]_{d+a})\\
&=&\HF_{\init_<(J)}(d+a)\\
&=&\HF_{J}(d+a)\\
&=&\dim_K(J_{d+a})\\
&=&\dim_K([\mm^dJ]_{d+a})=\beta_0(\mm^dJ).
\end{eqnarray*}
\end{enumerate}
\end{proof}


\begin{theorem}\label{t1}
If $I$ is a componentwise linear ideal, $\init_<(I)$ is a square-free ideal and $\init_<(I_{<d>})=\init_<(I)_{<d>}$ for all $d\in\Z_+$, then $\init_<(I)$ is a componentwise linear ideal. Moreover, we have $\beta_{i,i+j}(\init_<(I))=\beta_{i,i+j}(I)$ for all $i,j$.
\end{theorem}
\begin{proof}
We denote by $h=\reg(I)$ the Castelnuovo-Mumford regularity of $I$. 
Since $\init_<(I)$ is a square-free ideal,  we have $\reg(I)=\reg(\init_<(I))$ by \cite[Corollary 2.7]{CV}. It follows that $\beta_{0,i}(\init_<(I))=0$ for all $i>h$. Furthermore, since $I$ is componentwise linear, $\beta_{0,h}(\init_<(I))\ge \beta_{0,h}(I)>0$. Hence $h$ is the highest degree of a generator in a minimal system of generators of $\init_<(I)$. We show that $\init_<(I)$ is componentwise linear by induction on $h$.
\begin{enumerate}
\item[$h=1$:] $\init_<(I)$ has $1$-linear resolution and so it is componentwise linear.
\item[$h>1$:] Since  $\init_<(I_{<d>})=\init_<(I)_{<d>}$ for all $d\in\Z_+$, $\init_<(I_{\le d})=\init_<(I)_{\le d}$ for all $d\in\Z_+$ by Lemma\til\ref{l2} $ii\Rightarrow iii$). For each $d\in\Z_+$, since
\begin{eqnarray*} 
(I_{\le h-1})_{\le d}=
\begin{cases}
I_{\le d}&\text{ if \ } d\le h-1,\\
I_{\le h-1}&\text{  if \ } d>h-1,
\end{cases}
\end{eqnarray*}
we have 
\begin{eqnarray*} \init_<((I_{\le h-1})_{\le d})&=&
\begin{cases}
 \init_<(I_{\le d})&\text{ if \ } d\le h-1,\\
\init_<(I_{\le h-1})&\text{  if \ } d>h-1
\end{cases}\\
&=& 
\begin{cases}
 \init_<(I)_{\le d}&\text{ if \ } d\le h-1,\\
\init_<(I)_{\le h-1}&\text{  if \ } d>h-1\end{cases}\\
&=&(\init_<(I)_{\le h-1})_{\le d}\\
&=& \init_<(I_{\le h-1})_{\le d}.
\end{eqnarray*}
Using again  Lemma\til\ref{l2} $iii\Rightarrow ii$) we have $\init_<((I_{\le h-1})_{<d>})=\init_<(I_{\le h-1})_{<d>}$ for all $d\in\Z_+$.
Moreover, since $I$ is componentwise linear, $I_{\le h-1}$ is componentwise linear and $\reg(I_{\le h-1})=h-1$. Since  $\init_<(I)$ is a square-free ideal,  $\init_<(I_{\le h-1})=\init_<(I)_{\le h-1}$  is also a square-free ideal and  $\reg(\init_<(I_{\le h-1}))=h-1$ by \cite[Corollary 2.7]{CV}.
Hence $\init_<(I_{\le h-1})$ is componentwise linear by inductive hypothesis.  It follows that
$$\init_<(I_{<h-1>})=\init_<((I_{\le h-1})_{<h-1>})=\init_<(I_{\le h-1})_{<h-1>}$$ 
has $(h-1)$-linear resolution and so  $\mm\init_<(I)_{<h-1>}=\mm\init_<(I_{<h-1>})$ has $h$-linear resolution. \\
Now we consider the following two short exact sequences
$$0\la \mm\init_<(I)_{<h-1>}\la \init_<(I)_{<h>} \la M_h\la 0$$
and 
$$0\la\init_<(I)_{\le h-1}\la\init_<(I)\la M_h\la 0,$$
where $M_h=\init_<(I)_{<h>}/\mm\init_<(I)_{<h-1>}$.
For each $i,j$ they yield the following long exact sequences
\begin{eqnarray*}
\dots&\la&\Ext^{i-1}_R(\mm\init_<(I)_{<h-1>}, K)_{i+j}\la \Ext^{i}_R(M_h, K)_{i+j}\la \\
&&\Ext^{i}_R(\init_<(I)_{<h>}, K)_{i+j}\la \Ext^{i}_R(\mm\init_<(I)_{<h-1>}, K)_{i+j}\la\dots
\end{eqnarray*}
and
\begin{eqnarray*}
\dots&\la&\Ext^{i-1}_R(\init_<(I)_{\le h-1}, K)_{i+j}\la \Ext^{i}_R(M_h, K)_{i+j}\la \\
&&\Ext^{i}_R(\init_<(I), K)_{i+j}\la \Ext^{i}_R(\init_<(I)_{\le h-1}, K)_{i+j}\la\dots.
\end{eqnarray*}
Since $\beta_{i-1,i+j}(\mm\init_<(I)_{<h-1>})=\beta_{i,i+j}(\mm\init_<(I)_{<h-1>})=0$ for each $j>h$,  $$\Ext^{i-1}_R(\mm\init_<(I)_{<h-1>}, K)_{i+j}=\Ext^{i}_R(\mm\init_<(I)_{<h-1>}, K)_{i+j}=0$$ for each $j>h$. This implies
$$\Ext^{i}_R(M_h, K)_{i+j}\iso \Ext^{i}_R(\init_<(I)_{<h>}, K)_{i+j}$$ 
for each $i$ and for each $j>h$ by the first long exact sequence.
Since  $\reg(\init_<(I)_{\le h-1})=\reg(\init_<(I_{\le h-1}))=h-1$, we have
\begin{eqnarray*}
\dim_K(\Ext^{i-1}_R(\init_<(I)_{\le h-1}, K)_{i+j})&=&\beta_{i-1,i+j}(\init_<(I_{\le h-1})) =0
\end{eqnarray*}
and
\begin{eqnarray*}
\dim_K(\Ext^{i}_R(\init_<(I)_{\le h-1}, K)_{i+j})&=&\beta_{i,i+j}(\init_<(I_{\le h-1})) =0
\end{eqnarray*}
for all $i$ and for all $j\ge  h$, and so
$$\Ext^{i}_R(M_h, K)_{i+j}\iso\Ext^{i}_R(\init_<(I), K)_{i+j}$$
for all $i$ and for all $j\ge h$ by the second long exact sequence.
Therefore, for all $i$ and for all $j> h$,
$$\Ext^{i}_R(\init_<(I)_{<h>}, K)_{i+j}\iso\Ext^{i}_R(\init_<(I), K)_{i+j}.$$
We have
\begin{eqnarray*}
\beta_{i,i+j}(\init_<(I)_{<h>})&=&\dim_K(\Ext^{i}_R(\init_<(I)_{<h>}, K)_{i+j})\\
&=&\dim_K(\Ext^{i}_R(\init_<(I), K)_{i+j})\\
&=&\beta_{i,i+j}(\init_<(I))\\
&=&0
\end{eqnarray*}
for all $i$ and for all $j> h=\reg(\init_<(I))$.
Since $\init_<(I)_{<h>}$ is generated by generators of degree $h$, $\beta_{0,j}(\init_<(I)_{<h>})=0$ for $j<h$, and so $\beta_{i,i+j}(\init_<(I)_{<h>})=0$ for  $j<h$. 
It follows that  $\beta_{i,i+j}(\init_<(I)_{<h>})=0$ for all $j\not=h$.\\
By inductive hypothesis and by Lemma\til\ref{l2}, $\init_<(I)_{\le h-1}$ is componentwise linear. Hence for each $d\le h-1$ and for each $j\not=d$,
$$\beta_{i,i+j}(\init_<(I)_{<d>})=\beta_{i,i+j}((\init_<(I)_{\le h-1})_{<d>})=0.$$ Therefore, $\beta_{i,i+j}(\init_<(I)_{<d>})=0$ for each $d$ and for each $j\not=d$, that is, $\init_<(I)$ is a componentwise linear ideal.
\end{enumerate}
Furthermore,  
by Lemma\til\ref{l2} $ii\Rightarrow i$) and Lemma\til\ref{l3} $i$)  we get $\init_<(\mm I_{<d>})= \mm\init_<(I_{<d>})=\mm \init_<(I)_{<d>}$ for each $d$. If both $I$ and $\init_<(I)$ are componentwise linear, then both $I_{<d>}$ and $\init_<(I_{<d>})=\init_<(I)_{<d>}$ have $d$-linear resolutions for each $d$. It follows that both  $\mm I_{<d>}$ and $\init_<(\mm I_{<d>})$ have $(d+1)$-linear resolutions for all $d$. Using their Hilbert series we obtain 
$\beta_{i,i+j}(I_{<d>})=\beta_{i,i+j}(\init_<(I_{<d>}))$ and
$\beta_{i,i+j}(\mm I_{<d>})=\beta_{i,i+j}(\init_<(\mm I_{<d>}))$ for all $i,j$.
Therefore, by \cite[Proposition $1.3$]{HH} we have that
\begin{eqnarray*}
\beta_{i,i+j}(\init_<(I))&=&\beta_{i}(\init_<(I)_{<j>})-\beta_{i}(\mm \init_<(I)_{<j-1>})\\
&=&\beta_{i}(\init_<(I_{<j>}))-\beta_{i}(\init_<(\mm I_{<j-1>}))\\
&=&\beta_{i}(I_{<j>})-\beta_{i}(\mm I_{<j-1>})\\
&=&\beta_{i,i+j}(I)
\end{eqnarray*}
for each $i,j$.
\end{proof}

Now we make a short discussion on the necessity of the assumptions of the above theorem.
The following example shows that assumption ``$\init_<(I)$ is a square-free ideal'' is a necessary condition for our result.
\begin{example}
Let $X=\begin{pmatrix}
a & b & c\\
b & d & e\\
c & e &f
\end{pmatrix}$ be a symmetric matrix and  let $<$ be the graded reverse lexicographic order on $K[a,b,c,d,e,f]$ induced by $a>b>c>d>e>f.$   
The ideal $I$ generated by the $2$-minors of $X$ is
\begin{eqnarray*}
I &=&(-b^2+ad, -bc+ae,-cd+be, -c^2+af, -ce+bf, -e^2+df).
\end{eqnarray*}
Using Macaulay2 \cite{M2} we compute  the  Betti table of $R/I$: 
\begin{center}
\begin{tabular}{ |c|c c c c | } 
 \hline
 & 0 & 1 & 2&3 \\ 
 \hline
0 & 1 & 0 & 0 &0  \\ 
1 & 0& 6 & 8&3\\ 
 \hline
\end{tabular}
\end{center}
Notice that $I$ has $2$-linear resolution so it is a componentwise linear ideal.
Using again Macaulay2 \cite{M2} we obtain that the initial ideal of $I$ with respect to $<$
\begin{eqnarray*}
\init_<(I)&=&(e^2,ce, cd,c^2, bc, b^2)
\end{eqnarray*}
is not a square-free ideal. 
And according to the  Betti table of $R/\init_<(I)$:
\begin{center}
\begin{tabular}{ |c|c c c c c| } 
 \hline
 & 0 & 1 & 2&3&4 \\ 
 \hline
0 & 1 & 0 & 0 &0 &0 \\ 
1 & 0& 6 & 8&4&1 \\ 
2 &0 &0&1&1& 0\\ 
 \hline
\end{tabular}
\end{center}
we have that $\init_<(I)$ is not a componentwise linear ideal even if $\beta_0(I)=\beta_0(\init_<(I))=6$.
\end{example}

But the necessity of assumption ``$\init_<(I_{<d>})=\init_<(I)_{<d>}$ for all $d\in\Z_+$'' is still an open question. Using Lemma\til\ref{l2}, this question can be reduced to the following one:
 \begin{question*}
If $I$ is a componentwise linear ideal such that $\init_<(I)$ is square-free, can we obtain that $\init_<(I_{\le {h-1}})$ is also a square-free ideal with $h=\reg(I)$?
\end{question*}
If the answer to the above question is true, then the necessity of this assumption can be denied.


\section{Some Applications}
In this section, we study some applications of  Theorem\til\ref{t1}. First of all, notice that the first part of the proof of \cite[Theorem 5.4]{CaV} showed that
if $\beta_0(I)=\beta_0(\init_<(I))$, then for each $d\in\Z_+$ the initial ideal $\init_<(I_{<d>})$ is generated in degree $d$, and so $\init_<(I_{<d>})=\init_<(I)_{<d>}$ for all $d\in\Z_+$. Therefore, our main theorem has the following consequence:

\begin{corollary}\label{main}
Assume that $\init_<(I)$ is square-free and $\beta_0(I)=\beta_0(\init_<(I))$. Then $I$ is a componentwise linear ideal if and only if  $\init_<(I)$ is a componentwise linear ideal. In particular,  if furthermore one of $I$ and $\init_<(I)$ is  componentwise linear, we have $\beta_{i,i+j}(\init_<(I))=\beta_{i,i+j}(I)$ for all $i,j$.
\end{corollary}
\begin{proof}
It follows by Theorem\til\ref{t1} and \cite[Theorem 5.4]{CaV}.
\end{proof}

In what follows, we suppose furthermore that   $w=(w_1, \dots, w_n)\in\N^n$ is a weight vector and $N$ is a \fg\ $R[t]$-module such that it is a graded $K[t]$-module and it is flat over $K[t]$. We denote by $\hom_w(J)$ the $w$-homogenization of  an ideal $J\se R$ with respect to $w$.\\

Let us recall that $S=R[t]/\hom_w(I)$ is  $N$-fiber full up to  an integer $h$  as  an $R[t]$-module if,  for any $m\in\Z_{+}$, the natural projection $S/t^mS\la S/tS$ induces injective maps $\Ext^i_{R[t]}(S/tS, N) \la \Ext^i_{R[t]}(S/t^mS, N)$ for all $i\le h$ (see \cite[Definition 1.1]{Yu}). One result related to this notion is the following (see \cite[Corollary 3.2]{Yu}): 
\begin{quote}
If $R[t]/\hom_w(I)$ is  $N$-fiber full up to  $h$  as  an $R[t]$-module, then $$\dim_K(\Ext^i_R(R/I,N/tN)_j) = \dim_K(\Ext^i_R(R/\init_w(I),N/tN)_j)$$ for all $i\le h-2$ and for all $j\in\Z$.
\end{quote}
Actually the converse of the above result also holds and we have the following one:

\begin{remark}\label{nff}
Let $J\se R$ be an ideal.
Then
\begin{itemize}
\item[$i$)]
$S=R[t]/\hom_w(J)$ is $N$-fiber-full  up to $h$ as an $R[t]$-module
if and only if $\Ext^i_{R[t]}(S,N)$ is a flat $K[t]$-module for $i\le h-1$. 
\item[$ii$)]
If furthermore $J$ is homogeneous, then
$S$ is $N$-fiber-full  up to $h$ as an $R[t]$-module if and only if
$$\dim_K(\Ext^i_R(R/J,N/tN)_j)=\dim_K(\Ext_R^i(R/\init_w(J), N/tN)_j)$$
for all $i\le h-2$ and for all $j\in\Z$. \\
In particular, 
\begin{itemize}
\item[$iii$)]
if $N=R[t]$, then $S$ is $R[t]$-fiber-full  up to $h$ as an $R[t]$-module if and only if
$$\dim_K(\H_\mm^i(R/J)_j)=\dim_K(\H_\mm^i(R/\init_w(J))_j)$$
for all $i\ge n-h+2$ and for all $j\in\Z$.
\item[$iv$)]
If $N=K[t]$, then $S$ is $K[t]$-fiber-full  up to $h$ as an $R[t]$-module if and only if
$$\beta_{i,j}(R/J)=\beta_{i,j}(R/\init_w(J))$$
for all $i\le h-2$ and for all $j\in\Z$.
\end{itemize}
\end{itemize}
\end{remark}
\begin{proof}
The part $i$) is a direct consequence of \cite[Theorem 2.6]{Yu}.
For the part $ii$) we only have to notice that the converse  of some steps of  the proof of \cite[Corollary 3.2]{Yu} are also true.  More precisely, for the same reason discussed in \cite[Corollary 3.2]{Yu}   we have that for each $i,j\in\Z$, 
$$\bigoplus_{l\in\Z}\Ext^i_{R[t]}(S,N)_{(j, l)}\iso K[t]^{a_{i,j}}\oplus(\bigoplus_{k\in\Z_+}(K[t]/(t^k))^{b_{i,j,k}})$$
for some natural numbers $a_{i,j}$ and $b_{i,j,k}$. 
And for every $i,j\in\Z$,
$$\dim_K(\Ext^i_R(R/J,N/tN)_j)=a_{i,j},$$
$$\dim_K(\Ext_R^i(R/\init_w(J), N/tN)_j)=a_{i,j}+b_{i,j}+b_{i+1,j},$$
where $b_{i,j}=\sum_{k\in\Z_+}b_{i,j,k}$. 
Therefore, $S$ is $N$-fiber-full up to $h$ if and only if $\Ext^i_{R[t]}(S, N)$ is a flat $K[t]$-module for all $i\le h-1$, if and only if $b_{i,j, k}=0$ for all $i\le h-1$ and for all $j,k$, if and only if $b_{i,j}=b_{i+1,j}=0$ for all $i\le h-2$ and for all $j$, if and only if $\dim_K(\Ext^i_R(R/J,N/tN)_j)=a_{i,j}= \dim_K(\Ext_R^i(R/\init_w(J), N/tN)_j)$  for all $i\le h-2$ and for all j.\\
In particular,  if $N=R[t]$, then $N/tN\iso R$ and the part $iii$) is obtained by using the local duality theorem for graded modules (see \cite[Theorem 3.6.19]{CMrings}).\\
If $N=K[t]$, then $N/tN\iso K$. By \cite[Proposition 1.3.1]{CMrings}  we have $\dim_K(\Ext^i_R(R/J,K)_j)=\beta_{i,j}(R/J)$ and $ \dim_K(\Ext_R^i(R/\init_w(J), K)_j)=\beta_{i,j}(R/\init_w(J))$ for all $i,j$.
\end{proof}

Now considering the fact (see \cite[Proposition 3.4]{levico})
\begin{quote}
Given an ideal $J\se R$ and given a monomial order $<$ on $R$,  there exists a suitable weight vector $w$ such that $\init_w(J) =\init_<(J)$.
\end{quote}
we show that the above remark and Corollary\til\ref{main} imply the following result:

\begin{proposition}
Assume that one of the following two conditions holds
\begin{enumerate}
\item[$i$)]  $I$ is a componentwise linear ideal and $\init_<(I)$ is square-free,
\item[$ii$)]  $\init_<(I)$ is a componentwise linear ideal.
\end{enumerate}
Then $S=R[t]/\hom_w(I)$ is $K[t]$-fiber-full  up to $3$ if and only if $S$ is $K[t]$-fiber-full  up to $h$ for all $h\in\Z$, where $w$ is a weight vector such that $\init_w(I) =\init_<(I)$.
\end{proposition}
\begin{proof}
One implication is trivial. On the other hand, if $S$ is $K[t]$-fiber-full  up to $3$, then $\beta_{1,j}(R/I)=\beta_{1,j}(\init_<(I))$ for all $j\in\Z$ by Remark\til\ref{nff} $iv)$, and it follows that $\beta_{0}(I)=\beta_{0}(\init_<(I))$. If one of the two conditions of our assumption holds,  we obtain that $\beta_{i,i+j}(I)=\beta_{i,i+j}(\init_<(I))$ for all $i,j$ by Corollary\til\ref{main} and by \cite[Theorem 5.4]{CaV}. This implies that $S$ is $K[t]$-fiber-full  up to $h$ for all $h\in\Z$ by using again Remark\til\ref{nff} $iv)$.
\end{proof}

\end{document}